\documentclass[12pt,a4paper]{amsart}
\usepackage{pst-node}
\usepackage{tikz}
\usetikzlibrary{calc, arrows, decorations.markings}
\usetikzlibrary{matrix,arrows}
\usepackage{blindtext}
\usepackage{graphicx} % for \rotatebox

\usepackage[mathscr]{eucal}
\usepackage{amssymb,amsmath,scalefnt}
\usepackage{mathrsfs}
\usepackage{amsfonts}
\usepackage[all]{xy}%\usepackage{tikz-cd}
\usepackage{xypic,color}
\usepackage{a4}
\DeclareMathAlphabet{\mathpzc}{OT1}{pzc}{m}{it}

\newtheorem{proposition}{\sc Proposition}[section]
\newtheorem{lemma}[proposition]{\sc Lemma}
\newtheorem{corollary}[proposition]{\sc Corollary}
\newtheorem{theorem}[proposition]{\sc Theorem}

\theoremstyle{definition}

\theoremstyle{remark}

\renewcommand{\[}{\begin{equation}}
\renewcommand{\]}{\end{equation}}

\begin{document}
\baselineskip=15pt
%\vspace*{-15mm}
\title[The cyclic-homology Chern-Weil homomorphism  for principal coactions]{The cyclic-homology Chern-Weil homomorphism\\  
for principal coactions}
\author{Piotr~M.~Hajac}
\address{Instytut Matematyczny, Polska Akademia Nauk, ul.~\'Sniadeckich 8, 00--656 Warszawa, Poland}
\email{pmh@impan.pl}
\author{Tomasz Maszczyk}
\address{Instytut Matematyki, 
     Uniwersytet Warszawski,
ul.\ Banacha 2,
02--097 Warszawa, Poland}
\email{t.maszczyk@uw.edu.pl}
\begin{abstract}\normalsize
We view the space of cotraces in the structural coalgebra of a principal coaction as a noncommutative counterpart of the classical Cartan model.
Then we define  the  cyclic-homology Chern-Weil 
homomorphism  by extending the Chern-Galois character from the characters of finite-dimensional comodules  to arbitrary cotraces.
To reduce the  cyclic-homology Chern-Weil 
homomorphism to a tautological natural transformation, we
replace  the unital coaction-invariant subalgebra by its certain natural
H-unital nilpotent extension (row extension), and prove that their cyclic-homology groups are isomorphic.
In the proof, we use a chain homotopy invariance of complexes computing Hochschild, and hence cyclic homology,
for arbitrary row extensions. In the  context of the cyclic-homology Chern-Weil 
homomorphism, a row extension is provided by the Ehresmann-Schauenburg quantum groupoid with a nonstandard multiplication. 
\color{black}
\end{abstract}
\maketitle
\vspace*{-10mm}
{\tableofcontents}

\section{Introduction}

 A formula computing the Chern character of a finitely generated projective  module associated with a given finite-dimensional representation 
\cite{bh04} uses a strong connection \cite{h-pm96,bh04} and multiple comultiplication applied to the character of
representation  to produce a  cycle in the 
complex computing cyclic homology of the algebra of invariants. Its homology class is called the \emph{Chern-Galois character} \cite{bh04} of
representation. It  is a 
fundamental tool in calculating $K_0$-invariants of modules associated to principal coactions of coalgebras on algebras, in particular to principal 
comodule algebras  in Hopf-Galois theory \cite{hkmz11}.

The goal of this  paper is to factorize the Chern-Galois character through a noncommutative Chern-Weil homomorphism
taking values in  a model of cyclic homology reducing the homomorphism to a tautological natural transformation.
First, we achieve a natural factorization of the Chern-Galois character by replacing the unital coaction-invariant subalgebra by its certain natural
H-unital nilpotent extension, which we call a \emph{row extension}. Next, we observe that
 the Chern-Galois character extends from the characters of finite-dimensional comodules to arbitrary cotraces 
while still producing elements of the
cyclic homology of the row extension
stable under Connes' periodicity operator. Since the space of cotraces can be viewed as a noncommutative replacement of the Cartan model, 
and the cyclic homology of the row extension turns out to be isomorphic to the cyclic homology 
of the unital coaction-invariant subalgebra (playing the role of the de Rham cohomology of the base space), 
we  interpret this extension as a cyclic-homology counterpart of the
classical \emph{Chern-Weil homomorphism}.
Although our Chern-Weil homomorphism can be  obtained simply as an extension of the Chern-Galois character to all cotraces
without referring to row extensions, we
need the row-extension model of  the cyclic homology of the base-space algebra to manifest the Chern-Weil
homomorphism as a tautological natural transformation.

An abstract argument used to achieve the above goal can also be applied to matrix projections to produce
 the Chern character from K-theory to cyclic homology. Both cases are instances of a common construction we call \emph{abstract Chern-type character}. 

Another remarkable common feature of these two constructions is that they both can be defined in a tautological way in terms of a canonical block- matrix H-unital algebra. In the well-known Chern-character  case, it is an algebra of infinite matrices with entries in a given algebra. In our Chern-Weil case this  H-unital algebra is a specific Hochschild extension of an 
algebra coming from a module equipped with a module map to the algebra, which we call augmentation. We prove that  every such an extension is 
isomorphic to to the block-matrix algebra whose the only possibly nonzero row consists of the algebra itself followed by the kernel of the 
augmentation. We call the latter \emph{row extensions} and prove that the Hochschild homology is invariant under such extensions by providing an 
explicit homotopy equivalence of complexes. 

In the case of faithfully flat Hopf-Galois extensions, the corresponding augmented module is the Ehresmann-Schauenburg quantum groupoid with 
the augmentation being its counit. The kernel of the augmentation is therefore equal to invariant universal noncommutative differential forms. This 
determines the canonical block-matrix structure of that row extension completely.

The fact that the space  of cotraces could be understood as a \emph{cyclic-homology Cartan 
model} of a conjectural cyclic homology of the classifying space of the coalgebra $C$ 
we justify by the graded-space construction associated  with the Ad-invariant $\mathfrak{m}$-adic filtration on class functions, which  produces the classical space of Ad-invariant polynomials on the Lie algebra.   

 Since also the abelian group completion ${\rm Rep}(C)$ of the monoid of finite dimensional $C$-comodules could be understood as a conjectural  
$K_{0}$-group of the classifying space of the coalgebra $C$, the formula for the Chern-Galois character from \cite{bh04} could be 
understood as conjectural naturality of the Chern character under the classifying map for a noncommutative principal bundle corresponding 
to a principal $C$-Galois extension $B\subseteq A$. All this can be subsumed by the following commutative diagram 
 $$\xymatrixcolsep{4em}\xymatrix{
{\rm Rep}(C) \ar[d]_{\chi} \ar[r]^-{[A\Box^{C}(-)]} &{\rm K}_{0}(B)\ar[d]^{{\rm ch}_{n}}\\
C^{\rm tr} \ar[r]^-{{\rm chw}_{n}} &{\rm HC}_{2n}(B)},
$$
where the map $[A\Box^{C}(-)]$ associating a finitely generated projective module with a given representation should be understood as the map induced by a classifying map on K-theory, the character $\chi$ of a representation should be understood as the Chern character for the classifying space and the cyclic Chern-Weil map ${\rm chw}$ should be understood as the map induced by a classifying map on cyclic homology. The Chern-Galois is the diagonal composite in this diagram.

The above commutative diagram can be understood as a noncommutative counterpart of naturality of the Chern character under the classifying map ${\rm cl}: Y\rightarrow BG$ of a $G$-principal bundle $X\rightarrow  Y$ corresponding to a principal $G$-action $X\times G\rightarrow  X$ with the space of orbits $Y=X/G$, tantamount to commutativity of the following diagram
 $$\xymatrixcolsep{5em}\xymatrix{
{\rm K}^{0}(BG)\ar[d]_{{\rm ch}_{n}(BG)} \ar[r]^-{{\rm K}^{0}({\rm cl})} &{\rm K}^{0}(Y)\ar[d]^{{\rm ch}_{n}(Y)}\\
{\rm H}^{2n}(BG) \ar[r]^-{{\rm H}^{2n}({\rm cl})} &{\rm H}^{2n}(Y). }
$$  

The above analogy between the role of block-matrix $H$-unital algebras in the construction of both Chern and Chern-Galois characters can be subsumed in the following commutative diagram

 $$\xymatrixcolsep{4pc}\xymatrix{
{\rm Rep}(C) \ar[dd]_{\widetilde{\rm chg}_{n}(\ell)}\ar[ddr]^{{\rm chg}_{n}(\ell)} \ar[rr]^{[A\Box^{C}(-)]}& &{\rm K}_{0}(B)\ar[dd]^{\widetilde{\rm ch}_{n}}\ar[ddl]_{{\rm ch}_{n}}\\ \\
{\rm HC}_{2n}(M) \ar[r]^{{\rm HC}_{2n}(\varepsilon)}_{\cong}&{\rm HC}_{2n}(B) &{\rm HC}_{2n}({\rm M}_{\infty}(B)).\ar[l]^{\cong}_{[{\rm tr}_{2n}]}}$$

Here the bottom horizontal arrows are isomorphisms of block-matrix $H$-unital models of cyclic homology of $B$ and vertical arrows are tautological constructions. Of course the right side factorization of the Chern character is very well known, and provides an analogy to the left           
side factorization of the Chern-Galois character. It should be stressed that the Chern-Galois character on representations and the Chern character on K-theory are abstract cyclic-homology Chern-type characters  for completely different reasons. Therefore it is a quite remarkable fact that there exists a construction \cite{bh04} of a matrix idempotent representing an associated finitely generated projective $B$-module out of a given representation and the strong connection relating these two constructions.

Another aspect of our construction consists in the fact that we work with complexes up to chain homotopy equivalence rather than with homology classes. It is motivated by the fact that  
although on the theoretical level the aforementioned invariance of Hochschild homology under row extensions can be established by the Wodzicki excision argument \cite{wodz88, wodz89}, a problem of making this argument  explicit in the resulting inverse excision isomorphism, as signalled in \cite{braun}, arises. We overcome this difficulty by constructing an explicit  homotopy compatible with an analogue of the filtration from \cite{gu-gu-96}, providing a homotopy equivalence of corresponding complexes. The fact that  all homotopies we use are natural and explicit suggests a higher homotopy landscape behind our construction, according to the ideas surveyed in \cite{hub}.  Our Lemma \ref{main lemma}, replacing here the Homological Perturbation Theory evoked in \cite{hub}, could be of independent interest. Similarly as Homological Perturbation Theory  is used as a  tool in computing Hochschild and cyclic homology and the Chern character \cite{lam, alv-arm-re-sil, kass}, we use our Lemma \ref{main lemma} in calculations in the homotopy category of chain complexes. An additional substantiation of homotopical approach is the fact that it is a natural environment for the classical Chern-Weil theory~\cite{free-hop}.           

To put our construction in historical perspective, let us recall other approaches to the Chern-Weil map in noncommutative geometry and compare them with ours. As it seems, the first instance of a connection between cotraces and Chern-Weil theory goes back to Quillen's work \cite{qui}. Although the coalgebra there is the bar construction of an algebra, the analogy with the Chern-Weil homomorphism is explicitly stressed therein. 

Next, in \cite{alex-mein, mein} Alexeev and Meinrenken introduced  noncommutative Chern-Weil theory based on a  specific noncommutative deformation of the classical Weil model aiming to extend the Duflo isomorphism for quadratic Lie algebras to the level of equivariant cohomology. However, instead arbitrary Hopf algebras they work only with the universal enveloping algebra of a Lie algebra, and without referring to cyclic homology.   

In \cite{crai02} Crainic considers  a Weil model in the context of Hopf-cyclic homology of Hopf algebras. However, his characteristic map based on the characteristic map of Connes and Moscovici takes values in the cyclic homology of a Hopf-module algebra  instead of the cyclic homology of the algebra of coaction invariants. As such, it cannot be a noncommutative counterpart of the classical Chern-Weil homomorphism.

\section{Homotopy category of chain complexes}

\subsection{Killing contractible complexes} The following lemma should have been proved sixty years ago. Strangely enough, the first approximation to it can be found in Loday's book without any further reference, under the name ``Killing contractible complexes'' \cite{lod}. Regretfully, the claim there is about a quasiisomorphism only instead of homotopy equivalence. Moreover, that quasiisomorphism doesn't respect the obvious structure of the short exact sequence of complexes. The homotopy equivalence was  achieved by Crainic only in 2004 \cite{crai} by constructing the explicit  homotopy inverse with use of the homological perturbation method. Still, his perturbed maps don't respect the obvious structure of the short exact sequence of complexes.  In contrast to these results, in our present approach we perturb neither the differential, nor the structure of the short exact sequence. Instead, we perturb a given splitting in the category of graded objects  to make it a splitting in the category of complexes, providing an explicit homotopy inverse. We focus on split short exact sequences of complexes since only those can produce distinguished triangles in the homotopy category of chain complexes.  
\begin{lemma}\label{main lemma}
Assume that    
$$\xymatrix{0 \ar[r] &X\ar[r]^{\iota} & Y \ar[r]^{\pi}&Z\ar[r]& 0}$$
is a short exact sequence of complexes in an abelian category split in the category of graded objects. Provided  $X$ is contractible, $\pi$ is a homotopy equivalence. 
\end{lemma}
\begin{proof} We will denote all differentials  by $d$ and all identity morphisms by $1$.

Consider a splitting 
$$\xymatrix{0 \ar[r] &X\ar[r]_{\iota} & Y \ar[r]_{\pi}\ar@{-->}@/_1pc/[l]_{\rho}&Z\ar@{-->}@/_1pc/[l]_{\sigma} \ar[r]& 0}$$
and a  homotopy $h$ contracting $X$. (The dashed arrows are not necessarily chain maps). This is tantamount to the following  identities.

\begin{minipage}{.5\textwidth} %
 \begin{align}
&d^2=0,\label{dd}\\
&d\iota =\iota d,\label{di}\\
&d\pi =\pi d,\label{dp}\\
&\pi\iota=0,\label{pi}
\end{align}
\end{minipage} %
\begin{minipage}{.5\textwidth} %
\begin{align}
&\pi\sigma=1,\label{ps}\\
&\rho\iota=1,\label{ri}\\
&\sigma\pi + \iota\rho=1,\label{sp+ir}\\
&\rho\sigma=0,\label{rs}\\
&hd+dh=1.\label{hd}
\end{align}
\end{minipage}

\vspace{1em}
Now we  define the following expressions
\begin{align}
\alpha:=\sigma d\pi,\label{alpha}\ \ \ \ \beta:=\iota\rho d\sigma\pi ,\ \ \ \ \gamma:=\iota d\rho,\ \ \ \ \widetilde{h}:=\iota h\rho,\ \ \ \ \widetilde{\sigma}:=(1-\widetilde{h}d)\sigma.
\end{align}\label{stilde}
By (\ref{ps}) and (\ref{pi}) we have
\begin{align}\label{right_inverse}
\pi\widetilde{\sigma}=1
\end{align}
which together with (\ref{rs}), (\ref{ri}) and (\ref{hd})
 implies that
\begin{align}\label{almost_chain_map}
(\alpha +\beta + \gamma)\widetilde{\sigma} - \widetilde{\sigma}d=\widetilde{h}(\beta\alpha+\gamma\beta).
\end{align}
Now, by (\ref{dp}),  (\ref{di}) and (\ref{sp+ir}) we have
\begin{align}\label{alpha+beta+gamma}
\alpha +\beta + \gamma=d.
\end{align}
After squaring both sides of (\ref{alpha+beta+gamma}) we use (\ref{dd}) on the right hand side, and on the left hand side we use the following identities:
\begin{align}
\alpha\beta=\alpha\gamma=\beta^2=\beta\gamma=0 \ \ \ \ \ &(\mbox{implied by (\ref{pi})}),\\
\gamma\alpha=0 \ \ \ \ \ &(\mbox{implied by (\ref{rs})}),\\
\alpha^2=0 \ \ \ \ \ &(\mbox{implied by (\ref{ps}) and (\ref{dd})}),\\
\gamma^2=0 \ \ \ \ \ &(\mbox{implied by (\ref{ri}) and (\ref{dd})}),
\end{align}
to obtain
\begin{align}\label{ba+gb}
\beta\alpha+\gamma\beta=0.
\end{align}
Therefore, after substituting  (\ref{alpha+beta+gamma}) and  (\ref{ba+gb}) to (\ref{almost_chain_map}) we obtain that 
\begin{align}\label{chain_map}
d\widetilde{\sigma} - \widetilde{\sigma}d=0,
\end{align}
i.e. $\widetilde{\sigma}$ is a chain map. 

Moreover, by  (\ref{ri}), (\ref{hd}) and (\ref{sp+ir})
\begin{align}\label{sp+gh+hb+hg}
\widetilde{\sigma}\pi +\gamma\widetilde{h} + \widetilde{h}(\beta+\gamma)=1.
\end{align}
However, using the following identities: 
\begin{align}\label{alpha/beta_h_tilde}
\alpha\widetilde{h}=0,\ \ \ \ \beta\widetilde{h}=0\ \ \ \ &(\mbox{implied by (\ref{pi})}),\\ 
\widetilde{h}\alpha=0\ \ \ \ &(\mbox{implied by (\ref{rs})}),
\end{align}
we can complete (\ref{sp+gh+hb+hg}) to
\begin{align}\label{almost_ho_inverse}
\widetilde{\sigma}\pi +(\alpha+\beta+\gamma)\widetilde{h} + \widetilde{h}(\alpha+\beta+\gamma)=1
\end{align}
which by (\ref{alpha+beta+gamma}) reads as
\begin{align}\label{ho_inverse}
\widetilde{\sigma}\pi +d\widetilde{h} + \widetilde{h}d=1.
\end{align}
Together with  (\ref{right_inverse}) the latter means that $\widetilde{\sigma}$ is a homotopy inverse to $\pi$.
\end{proof}
Since the above Lemma holds in any abelian category, an immediate consequence is its dual version.
\begin{lemma}
Assume that    
$$\xymatrix{0 \ar[r] &X\ar[r]^{\iota} & Y \ar[r]^{\pi}&Z\ar[r]& 0}$$
is a short exact sequence of complexes in an abelian category split in the category of graded objects. Provided  $Z$ is contractible, $\iota$ is a homotopy equivalence. 
\end{lemma}
\subsection{Other homotopy lemmas} The next lemmas are homotopy versions of some homological results collected in \cite{lod}. For the convenience of the reader we sketch their proofs by showing  the explicit homotopy inverses and homotopies as in \cite{lod}. 

We consider the first quadrant bicomplex ${\rm CC}(B\hspace{0.1em}|\hspace{0.1em}k)$ \cite{lod}  whose total complex computes cyclic homology, and the total complex of the sub-bicomplex ~${\rm CC}^{\{2\}}(B\hspace{0.1em}|\hspace{0.1em} k)$ consisting of the first two columns computes Hochschild homology of the $k$-algebra B  over a unital commutative ring $k$.
$$\xymatrix{
\ar[d]_{b}&\ar[d]_{-b'}&\ar[d]_{b} &\ar[d]_{-b'} & \\
B^{\hspace{0.03em}\otimes\hspace{0.03em} 3}\ar[d]_{b} & B^{\hspace{0.03em}\otimes\hspace{0.03em} 3}\ar[d]_{-b'}\ar[l]_{1-t} & B^{\hspace{0.03em}\otimes\hspace{0.03em} 3}\ar[d]_{b}\ar[l]_{N} & B^{\hspace{0.03em}\otimes\hspace{0.03em} 3}\ar[d]_{-b'}\ar[l]_{1-t}& \ar[l]_{N}\\
B^{\hspace{0.03em}\otimes\hspace{0.03em} 2}\ar[d]_{b} & B^{\hspace{0.03em}\otimes\hspace{0.03em} 2}\ar[d]_{-b'}\ar[l]_{1-t} & B^{\hspace{0.03em}\otimes\hspace{0.03em}2}\ar[d]_{b} \ar[l]_{N}& B^{\hspace{0.03em}\otimes\hspace{0.03em} 2}\ar[d]_{-b'}\ar[l]_{1-t}&\ar[l]_{N} \\
B                    & B\ar[l]_{1-t}                    & B\ar[l]_{N}                     & B \ar[l]_{1-t}                  & \ar[l]_{N}
}
$$
The following Lemma leads to a distinguished triangle in the homotopy category of complexes which (after applying the functor of homology) induces the long exact $ISB$-sequence relating Hochschild and cyclic homology \cite{lod}.
\begin{lemma}
The short exact sequence of total complexes  
\begin{align}
0\rightarrow {\rm Tot\ CC}^{\{2\}}(B\hspace{0.1em}\hspace{0.1em}|\hspace{0.1em}\hspace{0.1em}k)\rightarrow {\rm Tot\ CC}(B\hspace{0.1em}|\hspace{0.1em} k)\rightarrow {\rm Tot\ CC}(B\hspace{0.1em}|\hspace{0.1em} k)[2] \rightarrow 0
\end{align}
is graded-split and hence defines a distinguished triangle in homotopy category of complexes.
\end{lemma}
\begin{proof}
The graded splitting is obvious.
\end{proof}

The next Lemma enables, in the special case of our interest, a substancial simplification of the complex computing Hochschild homology to a complex ${\rm  CC}^{\{1\}}(B\hspace{0.1em}|\hspace{0.1em} k)$ consisting of the first column of  ${\rm  CC}(B\hspace{0.1em}|\hspace{0.1em} k)$.
\begin{lemma}
Provided  $B$ is left-unital, there is a graded-split short exact sequence of complexes with contractible kernel   
\begin{align}
0\rightarrow {\rm B}(B\hspace{0.1em}|\hspace{0.1em}k)\rightarrow {\rm Tot\ CC}^{\{2\}}(B\hspace{0.1em}|\hspace{0.1em} k)\rightarrow {\rm  CC}^{\{1\}}(B\hspace{0.1em}|\hspace{0.1em} k)\rightarrow 0
\end{align}
and hence a chain homotopy equivalence 
\begin{align}
{\rm Tot\ CC}^{\{2\}}(B\hspace{0.1em}|\hspace{0.1em} k)\rightarrow {\rm  CC}^{\{1\}}(B\hspace{0.1em}|\hspace{0.1em} k).
\end{align}
\end{lemma}
\begin{proof}
${\rm B}(B\hspace{0.1em}|\hspace{0.1em}k)$ is the bar-complex with the differential $b'$,  isomorphic up to a shift with the second column of ${\rm  CC}(B\hspace{0.1em}|\hspace{0.1em} k)$, admitting a contracting homotopy, defined with use of the left unit $1\in B$, of the following form as in \cite{lod}
\begin{align}
h(b^{0}\otimes\cdots\otimes b^{n})=1\otimes b^{0}\otimes\cdots\otimes b^{n}.
\end{align}
Since the graded splitting is obvious, the rest follows from Lemma \ref {main lemma}. 
\end{proof}

\begin{lemma}\label{matr infty}
If $B$ is unital the maps 
\begin{align}
{\rm inc}_{n}:{\rm C}_{n}(B \hspace{0.1em}|\hspace{0.1em} k)\rightarrow {\rm C}_{n}({\rm M}_{\infty}(B) \hspace{0.1em}|\hspace{0.1em} k) 
\end{align}
induced by the map of $k$-algebras
\begin{align}
{\rm inc}: B\rightarrow {\rm M}_{\infty}(B),\ \ \ \ 
b\mapsto   \left(\begin{array}{cc}b & 0\\
              0 & 0\end{array}\right)
\end{align}
form a homotopy equivalence of complexes.
\end{lemma}
\begin{proof}
Following \cite{lod} we take an obvious left inverse to ${\rm inc}_{n}$ of the form
\begin{align}
\begin{split}
 {\rm tr}_{n}:  {\rm C}_{n}({\rm M}_{\infty}&(B)\hspace{0.1em}|\hspace{0.1em} k)\rightarrow  {\rm C}_{n}(B\hspace{0.1em}|\hspace{0.1em} k),\\
{\rm tr}_{n}(\beta^0\otimes\beta^1\otimes\cdots\otimes \beta^n)&:=\sum_{i_{0},\ldots,i_{n}}\beta^0_{i_{0}i_{1}}\otimes\beta^1_{i_{1}i_{2}}\otimes\cdots\otimes \beta^n_{i_{n}i_{0}},
\end{split}
\end{align}
which is also a right inverse to ${\rm inc}_{n}$ up to the explicit homotopy 
\begin{align}
\begin{split}
h(\beta^0\otimes\beta^1\otimes\cdots\otimes \beta^n):=\sum_{m, i_{0},\ldots,i_{m}}(-1)^{m}E_{i_{0}1}(\beta^0_{i_{0}i_{1}})\otimes E_{11}(\beta^1_{i_{1}i_{2}})\otimes\cdots\\
\cdots\otimes E_{11}(\beta^r_{i_{m}i_{m+1}})\otimes E_{1i_{m+1}}(1)\otimes \beta^{m+1}\otimes\cdots\otimes\beta^n,
\end{split}
\end{align}
where $E_{ij}(b)$ denotes the elementary matrix with a single possibly non-zero entry $b\in B$.
\end{proof}
\begin{lemma}\label{conj ident}
 If  $B$ is a unital  the action of the group ${\rm GL}_{\infty}(B)$ on the algebra ${\rm M}_{\infty}(B)$ by conjugation
\begin{align}
\begin{split}
{\rm GL}_{\infty}(B)\times {\rm M}_{\infty}(B)&\rightarrow  {\rm M}_{\infty}(B),\\
(\gamma, \beta)&\mapsto \gamma\beta\gamma^{-1}
\end{split}
\end{align} 
induces a trivial action on  the object ${\rm C}({\rm M}_{\infty}(B)\hspace{0.1em}|\hspace{0.1em} k)$ of homotopy category of complexes.
\end{lemma}
\begin{proof}
The action of $\gamma$  on  ${\rm M}_{\infty}(B)$ by algebra automorphisms  is realized as simultaneous application of the two well defined actions: $\beta\mapsto \gamma\beta$ and $\beta\mapsto \beta\gamma^{-1}$. Since $k$ is unital commutative it induces the following well defined maps, the action on
 ${\rm C}({\rm M}_{\infty}(B)\hspace{0.1em}|\hspace{0.1em} k)$
 \begin{align}
\begin{split}
\gamma(\beta^0\otimes\beta^1\otimes\cdots\otimes \beta^n)=
\gamma\beta^0\gamma^{-1}\otimes\gamma\beta^1\gamma^{-1}\otimes\cdots\otimes \gamma\beta^n\gamma^{-1}
\end{split}
\end{align}
and a homotopy between the identity and that action
\begin{align}
\begin{split}
h(\beta^0\otimes\beta^1\otimes\cdots\otimes \beta^n)=\sum_{m=0 }^{n}
(-1)^{m}\beta^0\gamma^{-1}\otimes\gamma\beta^1\gamma^{-1}\otimes\cdots\\
\cdots\otimes \gamma\beta^m\gamma^{-1}\otimes \gamma\beta^{m+1}\otimes \beta^{m+2}\otimes\cdots\otimes\beta^n.
\end{split}
\end{align}
\end{proof}
\subsection{Abstract Chern-type  characters for cyclic objects}
Let us note now that for any cyclic object $X=(X_{m})$ in a category of modules we can consider sequences $x=(x_{m})$ satisfying the following two conditions when acted on by the cyclic operator 
${\rm t}$ and face operators ${\rm d}_{i}$
\begin{align}\label{abs tc}
{\rm t}(x_{m})&=(-1)^{m}x_{m},\\
\label{abs dc}
{\rm d}_{i}x_{m}&=x_{m-1}.
\end{align} 
Forming a module ${\mathcal K}(X)$, consisting of such sequences is a functor. 
For every element $x\in {\mathcal K}(X)$ we can construct a natural sequence of even chains in ${\rm Tot\ CC}(X)$ of the form
\begin{align}\label{abs chn}
{\rm ch}_{n}(x)=\sum_{m=0}^{2n}(-1)^{\lfloor m/2\rfloor}\frac{m!}{\lfloor m/2\rfloor !}x_{m},
\end{align} 
to obtain  a sequence of natural transformations.
\begin{proposition}\label{abs ch}
For every $x\in {\mathcal K}(X)$ the chains  ${\rm ch}_{n}(x)$ are cycles of degree $2n$ in ${\rm Tot\ CC}(X)$, whose cohomology classes  form a sequence stable under Connes periodicity operator $S$.
\end{proposition}
\begin{proof}
All formal arguments in the proof of \cite[Lemma-Notation 8.3.3]{lod} can be adapted to our situation. Namely,  by (\ref{abs dc}) followed by (\ref{abs tc})
\begin{align}\label{b}
{\rm b}(-2x_{2l})&=-2 x_{2l-1}=-(1-{\rm t})x_{2l-1},\\ \label{b'}
{\rm b}'(-lx_{2l-1})&=-l x_{2l-2}={\rm N} x_{2l-2},
\end{align}
which means that the chain ${\rm ch}_{n}(x)$ is a
cycle in ${\rm Tot\ CC}_{2n}(X)$. Finally, also the formal argument for stability under Connes' periodicity operator $S$ from the proof of \cite[Lemma-Notation 8.3.3]{lod} is still valid in our situation.
\end{proof}
We call the resulting natural transformation  
\begin{align}
{\rm ch}_{n}(X): {\mathcal K}(X)\rightarrow {\rm HC}_{2n}(X)
\end{align}
the \emph{abstract cyclic character}.

The motivating example  comes from the construction of the Chern character from matrix idempotents.

Let us recall the well known fact that the Chern character taking values in cyclic homology of an algebra $B$ goes in fact to cyclic homology of a nonunital algebra ${\rm M}_{\infty}(B)$ of infinite matrices. This is so because of a fundamental equivalence between isoclasses of finitely generated projective modules over $B$ and  ${\rm GL}_{\infty}(B)$-conjugacy classes of idempotents in  ${\rm M}_{\infty}(B)$. The fact that for a given idempotent $\mathbf{e}:=(e_{ij})\in {\rm M}_{\infty}(B)$ the sequence of  elements ${\rm c}_{m}:={\rm c}_{m}(\mathbf{e})$, where
\begin{align}\label{cne}
{\rm c}_{m}(\mathbf{e}):=\mathbf{e}\otimes \cdots\otimes \mathbf{e}\in {\rm M}_{\infty}(B)^{\otimes\ \! (m+1)},
\end{align}
satisfies the conditions (\ref{abs tc})-(\ref{abs dc}) follows immediately from the form of   ${\rm c}_{m}(\mathbf{e})$ and the idempotent property $\mathbf{e}^2=\mathbf{e}$. This by the abstract Chern-type character property means that the  chains 
\begin{align}
\widetilde{\rm ch}_{n}(\mathbf{e}):=\sum_{m=0}^{2n}(-1)^{\lfloor m/2\rfloor}\frac{m!}{\lfloor m/2\rfloor !}c_{m}(\mathbf{e})
\end{align}
are cycles in ${\rm Tot\ CC}_{2n}({\rm M}_{\infty}(B))$ and the sequence of their homology classes is stable under Connes' periodicity operator. 

Note that up to this point the construction of $\widetilde{\rm ch}$ is completely tautological. The next argument, identifying cyclic homology of  an 
H-unital algebra ${\rm M}_{\infty}(B)$ with cyclic homology of a unital algebra $B$ uses a specific homotopy equivalence of chain complexes as in 
Lemma~\ref{matr infty} and is well defined on the level of ${\rm K}_{0}(B)$ by virtue of Lemma \ref{conj ident}. 
 Namely,  applying the ${\rm GL}_{\infty}(B)$-conjugacy invariant map 
\begin{align}{\rm Tot\ CC}_{\bullet}({\rm M}_{\infty}(B))\rightarrow {\rm Tot\ CC}_{\bullet}(B),
\end{align}
induced by the map defined for all elements $ \mathbf{b}^{k}=(b^{k}_{ij})\in {\rm M}_{\infty}(B)$ as
\begin{align}
{\rm tr}_{n}: \mathbf{b}^{0}\otimes \cdots\otimes \mathbf{b}^{n}\mapsto \sum_{i_{0},\ldots, i_{n}}b^{0}_{i_{0}i_{1}}
\otimes\ldots\otimes   b^{n-1}_{i_{n-1}i_{n}}\otimes b^{n}_{i_{n}i_{0}},
\end{align}
to the element $\widetilde{\rm ch}_{n}(\mathbf{e})$ one gets the Chern character ${\rm ch}_{n}(\mathbf{e})$ 
depending only on the class in ${\rm K}_{0}(B)$ defined by the idempotent $\mathbf{e}\in {\rm M}_{\infty}(B)$.

Another example of an abstract cyclic-homology Chern-type character will come  from a construction of a cyclic-homology Chern-Weil 
homomorphism. A tautological construction on the level of an H-unital algebra with canonically isomorphic cyclic homology will need a class of  
another H-unital block-matrix algebra extension, which we introduce in the next section.

\section{Row extensions of unital algebras} 

All rings in this section are associative and possibly non-unital.
Let $k\rightarrow B$ be a ring homomorphism  and $\varepsilon: M\rightarrow B$ be a $(B,k)$-bimodule map from a  $(B,k)$-bimodule $M$ to the 
$k$-ring $B$. We call such a structure an \emph{augmented module} over a $k$-ring $B$. We define a $k$-ring structure on $M$ depending on this 
data as follows. As a $k$-bimodule it is the underlying left $k$-bimodule of the $(B,k)$-bimodule $M$ with the multiplication of elements of $M$ 
defined as a $k$-bimodule (in fact $(B,k)$-bimodule) map (the tensor product is balanced over $k$)
\begin{align}\label{mult}
 M\otimes M\rightarrow M,\ \ m\otimes m'\mapsto \varepsilon(m)m'.
\end{align}
By left $B$-linearity of $\varepsilon$ we have the identity
\begin{align}\label{assoc}
\varepsilon(\varepsilon(m)m')m''=\varepsilon(m)(\varepsilon(m')m'')
\end{align}
which amounts to associativity of (\ref{mult}).

%\subsection{The matrix structure of algebras from augmented modules}

\begin{proposition}\label{str}
The map $\varepsilon$ is a  $k$-ring map onto a left ideal $k$-subring $J$ in $B$, whose kernel is an  ideal $I$ in $M$ 
with zero right multiplication by elements of $M$.
In particular, $M$ is a Hochschild extension of $J$ by $I$,
\begin{align}\label{hoch}
0\rightarrow I\rightarrow M\rightarrow J\rightarrow 0.
\end{align}
\end{proposition}

\emph{Proof.} To prove that $\varepsilon$ is a  $k$-ring map we check that by left $B$-linearity of $\varepsilon$
\begin{align}\label{epsilonalgmap}
\varepsilon(\varepsilon(m)m')=\varepsilon(m)\varepsilon(m').
\end{align}
This implies that $I={\rm ker}(\varepsilon)$ is an ideal in $M$.
Since $\varepsilon$ is $(B,k)$-linear its image $J=\varepsilon(M)\subset B$ is a $(B,k)$-sub-bimodule isomorphic to $M/I$ via $\varepsilon$. By (\ref{mult}) $IM=0$, hence $I^2=0$ and $I$ becomes a $J$-bimodule such that $IJ=0$. Therefore $M$ is a Hochschild extension \cite{hochcoh, hochrel} of $J$ by $I$.\ $\Box$

\begin{proposition}\label{block}
Provided the surjective $(B,k)$-bimodule map $\varepsilon: M\rightarrow J$ admits a $k$-bimodule splitting, the $k$-ring $M$ is isomorphic to the $k$-bimodule $I\oplus J$ with multiplication
\begin{align}\label{cocmult}
(i,j)(i',j')=(ji'+\omega(j,j'), jj'),
\end{align}
where  $\omega: J\otimes_{k}J\rightarrow I$ is a $k$-bimodule map satisfying
\begin{align}\label{cocycle}
j\omega(j',j'')-\omega(jj', j'')+\omega(j,j'j'')=0.
\end{align}

If $J$ has a $k$-central right unit one can assume $\omega=0$, i.e.
\begin{align}\label{sdp}
M\cong \left(\begin{array}{cc}J & I\\
              0 & 0\end{array}\right).
\end{align}
\end{proposition}

\emph{Proof.} We will prove the proposition using a non-unital version of the relative Hochschild theory \cite{hochcoh, hochrel} of $k$-bimodule split extensions $M$ of $J$ by ideals $I$ satisfying $I^2=0$ under the stronger assumption that  $IM=0$.
A $k$-bimodule splitting of (\ref{hoch}) gives an isomorphism of $k$-bimodules $I\oplus J$ and amounts to a $k$-bimodule map $\sigma: J\rightarrow M$ such that $\varepsilon\circ \sigma = {\rm id}_{J}$. This defines $\omega(j,j'):=\sigma(j)\sigma(j')-\sigma(jj')$ which satisfies (\ref{cocycle}) by associativity of multiplication in $M$ and the property $IM=0$. It is in fact the Hochschild 2-cocycle condition missing one summand which vanishes by $IM=0$.

If $J$ has a $k$-central ($ec=ce$ for all $c\in k$) right unit $e\in J$ ($je=j$ for all $j\in J$) we can define a $k$-bimodule map $\lambda: J\rightarrow I$
\begin{align}\label{lambda}
\lambda(j):=-\omega(j,e)
\end{align}
and then putting $j''=e$ in (\ref{cocycle}), rewriting the result in terms of (\ref{lambda}) and adding the last summand being zero by $IJ=0$, we obtain
\begin{align}\label{cobound}
  \omega(j,j') &= \omega(j,j'e) = -j\omega(j',e)+ \omega(jj',e)\\
   & = j\lambda(j')-\lambda(jj')+\lambda(j)j'
\end{align}
which means that $\omega$ is a Hochschild coboundary of $\lambda$. By theory of Hochschild extensions this means that the automorphism
\begin{align}\label{aut}
(i,j)\mapsto (i+\lambda(j),j)
\end{align}
of the $k$-bimodule $I\oplus J$ transforms the multiplication (\ref{cocmult}) to the one with $\omega=0$. \ $\Box$

If $\varepsilon$ has a  $k$-bimodule section (in particular is surjective, i.e. $J=B$),  $B$ is unital and the  left $B$-module $I$ is free of rank $n$, $M$ is isomorphic to the algebra of $(n+1)\times (n+1)$ matrices over $B$, with nonzero entries concentrated at most in the first row, with $\varepsilon$ picking the first entry in that row. Motivated by this simple case we call  split $k$-ring extensions of the form 
\begin{align}\label{row}
M\cong \left(\begin{array}{cc}B & I\\
              0 & 0\end{array}\right),
\end{align}
 {\em row extensions}, even if $B$ is not right unital or $I$ is not free of finite rank as a  left $B$-module. For the further consideration it is crucial that such $M$ is a $k$-ring split extension, hence $M$ contains a copy of a $k$-ring $B$ as a $k$-subring with the ideal $I$ as its complement, and that $I$ is not merely a square zero ideal, but we have  $IM=0$. 
  Note that if the left unit $e$ of $B$  acts on the left $B$-module $M$ of the row extension as identity (then we say that $M$ is a unitary left module over a left unital $k$-ring $B$) $M$ becomes a left unital $k$-ring extension with the unit corresponding under the isomorphism (\ref{row}) to
\begin{align}\label{lu}
\left(\begin{array}{cc}e &0\\
              0 & 0\end{array}\right).
\end{align}
If $B$ is a $k$-algebra over a unital commutative ring $k$ we assume that all $k$-bimodules in question are symmetric (we refer to them simply as modules) and unitary.          
\subsection{Periodic cyclic homology of row extensions}
\begin{proposition}
The $k$-ring map $\varepsilon: M\rightarrow J$ induces an isomorphism of relative periodic cyclic homology of $k$-rings
\begin{align}\label{hp}
HP_{*}(M\hspace{0.1em}|\hspace{0.1em}k)\rightarrow HP_{*}(J\hspace{0.1em}|\hspace{0.1em}k).
\end{align}
\end{proposition}

\emph{Proof.} It is a consequence of the Goodwillie theorem \cite{good} (see Thm. 7.3 of \cite{cort} for the non-unital case) applied to the  Kadison ($k$-ring) periodic cyclic homology \cite{kadrel} of the $k$-ring extension (\ref{hoch}) by the nilpotent ideal $I$.\ $\Box$

\subsection{Hochschild complex of row extensions}
From now on we assume that the ground ring $k$ is a field, which we will supress in the notation.  Triangular 
$k$-algebras over a unital commutative ring $k$ are of the form
\begin{align}\label{triang}
T=\left(\begin{array}{cc}B &I\\
              0 & B'\end{array}\right),
\end{align}
where $B$, $B'$ are unital $k$-algebras and $I$ is a two-sided unitary $(B, B')$-bimodule (and symmetric as an underlying $k$-bimodule). The computation of  Hochschild homology of triangular $k$-algebras  is subsumed by \cite[Thm. 1.2.15]{lod} which says that the canonical $k$-algebra map $T\rightarrow B\times B'$ annihilating $I$ induces an isomorphism
\begin{align}\label{triang}
{\rm HH}_{*}(T)\cong{\rm HH}_{*}(B\times B').
\end{align}
Note that the pair of two projections onto the factors of the product $B\times B'$ induce a further isomorphism
\begin{align}\label{dir_sum}
{\rm HH}_{*}(B\times B')\cong{\rm HH}_{*}(B)\oplus {\rm HH}_{*}(B').
\end{align} 

Using the fact that a row extension fits into a ring extension of the form 
%could cover at least the $k$-algebra case of our Theorem \ref{HHrow}, by putting formally $B'=0$. However, for $B'=0$ the only right unitary $B'$ module (which condition is crucial in the proof of \cite[Thm. 1.2.15]{lod})  is trivial, hence it cannot cover the case of row extensions, unless $I=0$.  
\begin{align}
0\rightarrow  \left(\begin{array}{cc}B & I\\
              0 & 0\end{array}\right)\rightarrow  \left(\begin{array}{cc}B & I\\
              0 & k\end{array}\right)\rightarrow  \left(\begin{array}{cc}0 & 0\\
              0 & k\end{array}\right)\rightarrow 0,
\end{align}
which we shorten as ($T$ stands for a triangular $k$-algebra as above with $B'=k$) 
\begin{align}
0\rightarrow  M \rightarrow  T\rightarrow  k\rightarrow 0,
\end{align}
one can use Wodzicki's excision theorem \cite{wodz88, wodz89} for a left unital (hence H-unital) $k$-algebra $M$  to obtain a long exact sequence of Hochschild homologies, which by (\ref{triang}) and (\ref{dir_sum}) reads as

\begin{center} \begin{tikzpicture}[descr/.style={fill=white,inner sep=1.5pt}]
        \matrix (m) [
            matrix of math nodes,
            row sep=1em,
            column sep=2.5em,
            text height=1.5ex, text depth=0.25ex
        ]
        { {\rm HH}_{2n+1}(M) & {\rm HH}_{2n+1}(B)\oplus {\rm HH}_{2n}(k) & {\rm HH}_{2n+1}(k) \\
            {\rm HH}_{2n}(M) & {\rm HH}_{2n}(B)\oplus {\rm HH}_{2n}(k) & {\rm HH}_{2n}(k) \\
             {\rm HH}_{2n-1}(M) & {\rm HH}_{2n-1}(B)\oplus  {\rm HH}_{2n-1}(k) & {\rm HH}_{2n-1}(k). \\
        };

        \path[overlay,->, font=\scriptsize,>=latex]
        (m-1-1) edge (m-1-2)
        (m-1-2) edge (m-1-3)
        (m-1-3) edge[out=355,in=175,black] (m-2-1)
        (m-2-1) edge (m-2-2)
        (m-2-2) edge (m-2-3)
        (m-2-3) edge[out=355,in=175,black] (m-3-1)
        (m-3-1) edge (m-3-2)
         (m-3-2) edge (m-3-3);
\end{tikzpicture}
\end{center}
Since  in every row the first arrow goes into the first direct summand via the map induced by $\varepsilon$ and every second arrow is a projection 
onto a second direct summand, this proves that $\varepsilon$ induces a quasiisomorphism of Hochschild chain complexes.

Now we are  to promote it to a chain homotopy equivalence.

Using a $k$-module splitting $M=I\oplus B$ coming from a $k$-algebra splitting (\ref{row}), we can form split short exact sequences of 
symmetric $k$-bimodules ($\otimes$  stands for the $k$-balanced tensor product of symmetric $k$-bimodules)
\begin{align}
0\rightarrow \bigoplus_{p+q=n}M^{\otimes p}\otimes I
\otimes B^{\otimes q}\rightarrow M^{\otimes n+1}\stackrel{\varepsilon^{\otimes n+1}}{-\!\!\!-\!\!\!-\!\!\!\longrightarrow} 
B^{\otimes n+1}\rightarrow 0.    
\end{align}

Since $\varepsilon$ is a $k$-algebra map, and the multiplication of $M$ restricted to the image of $B$ under the $k$-algebra splitting inside $M$ 
coincides with the original multiplication of $B$,  the collection of induced maps $\varepsilon^{\otimes n+1}$ is a morphism of cyclic objects 
computing  Hochschild, cyclic, periodic  and negative cyclic homology of $k$-algebras. Now we show that for row extensions of left unital $k$-
algebras the induced map on  Hochschild chain complexes is a homotopy equivalence. It is another way to see that it induces an isomorphism on 
Hochschild homology, and hence by virtue of \cite[Prop. 5.1.6]{lod} this implies that the induced maps on cyclic, periodic  and negative cyclic 
homology are isomorphisms as well, alternative to the previous excision argument.

 \begin{theorem}\label{HHrow} Let $M$ be a left unitary row extension of a left unital $k$-ring $B$. 
Then the induced map of  Hochschild chain complexes 
\begin{align}\label{hh}
{\rm Tot\ CC}^{\{1\}}(M)\rightarrow {\rm Tot\ CC}^{\{1\}}(B)
\end{align}
is graded-split surjective and has a contractible kernel, in particular it is a chain homotopy equivalence.  
\end{theorem} 
\begin{proof} 
 We will show that the collection of maps 
\begin{align}
h: M^{\otimes p}\otimes I\otimes B^{\otimes q}&\rightarrow M^{\otimes p}\otimes I\otimes B^{\otimes q+1}\\
h(m_{1},\ldots, m_{p}, i, b_{1},\ldots, b_{q})& := (-1)^{p+1}(m_{1},\ldots, m_{p}, i, e, b_{1},\ldots, b_{q})
\end{align}
forms a homotopy contracting the kernel  of the map $\varepsilon^{\otimes n+1}$ of Hochschild complexes. 
Here, on the right hand side $e$ denotes the left unit of~$B$.  

The boundary on the kernel  induced by the Hochschild boundary, for $p+q=n$, reads as 
\begin{align}
b(m_{1},\ldots, m_{p}, i, b_{1},\ldots, b_{q})&= (b'(m_{1},\ldots, m_{p}), i,  b_{1},\ldots, b_{q})\nonumber\\
&-(-1)^{p}(m_{1},\ldots, m_{p-1},m_{p} i, b_{1},\ldots, b_{q})\nonumber\\
& -(-1)^{p}(m_{1},\ldots, m_{p}, i, b'(b_{1},\ldots, b_{q}))\nonumber\\
&+(-1)^{n}(b_{q}m_{1},\ldots, m_{p}, i, b_{1},\ldots, b_{q-1}),
\end{align}
where for any (not necessarily unital) associative $k$-ring $A$
\begin{align*}
b'(a_{1},\ldots, a_{r})&= (a_{1}a_{2},  a_{3},\ldots, a_{r})-(a_{1}, a_{2} a_{3},\ldots, a_{r})+\cdots +(-1)^{r}(a_{1},\ldots, a_{r-1}a_{r}).
\end{align*}
Note that here the condition $IM=0$ shortens the Hochschild boundary by one vanishing summand, as for the Hochschild cocycle encoding the structure of the extension in the proof of Proposition \ref{str}. 

Now, computing the two compositions $bh$ and $hb$ we get for $p+q=n$, $q>0$
\begin{align}
bh(m_{1},\ldots, m_{p}, i, b_{1},\ldots, b_{q})&=-(-1)^{p}(b'(m_{1},\ldots, m_{p}), i, e,  b_{1},\ldots, b_{q})\\
&+(m_{1},\ldots, m_{p-1}, m_{p}i, e, b_{1},\ldots, b_{q})\\
&+(m_{1},\ldots, m_{p}, i,  b_{1},\ldots, b_{q})\\
&-(m_{1},\ldots, m_{p}, i,  e, b'(b_{1},\ldots, b_{q}))\\
&+(-1)^{q}(b_{q}m_{1}, m_{2},\ldots, m_{p}, i, e, b_{1},\ldots, b_{q-1}),
\end{align} 

\begin{align}
hb(m_{1},\ldots, m_{p}, i, b_{1},\ldots, b_{q})&=(-1)^{p}(b'(m_{1},\ldots, m_{p}), i, e,  b_{1},\ldots, b_{q})\\
&-(m_{1},\ldots, m_{p-1}, m_{p}i, e, b_{1},\ldots, b_{q})\\
&+(m_{1},\ldots, m_{p}, i,  e, b'(b_{1},\ldots, b_{q}))\\
&-(-1)^{q}(b_{q}m_{1}, m_{2},\ldots, m_{p}, i, e, b_{1},\ldots, b_{q-1}),
\end{align}
and for $p+q=n$, $q=0$
\begin{align}
bh(m_{1},\ldots, m_{n}, i)&=-(-1)^{n}(b'(m_{1},\ldots, m_{n}), i, e)\\
&+(m_{1},\ldots, m_{n-1}, m_{n}i, e)\\
&+(m_{1},\ldots, m_{n}, i),
\end{align} 

\begin{align}
hb(m_{1},\ldots, m_{n}, i)&=(-1)^{n}(b'(m_{1},\ldots, m_{n}), i, e)\\
&-(m_{1},\ldots, m_{n-1}, m_{n}i, e),
\end{align} 
one sees that they add up to give $bh+hb=\rm{Id}$.
\end{proof}

\section{The noncommutative Chern-Weil homomorphism} 
%\subsection{Coalgebra-Galois extensions}

To understand  what follows, we refer to \cite{bh04} for the basic facts and definitions.
Let $A$ be a right comodule algebra for a coalgebra $C$ with a group-like element $e\in C$.  
We denote the $C$-coaction on $A$ with use of the Sweedler notation
\begin{align}
A\rightarrow A\otimes C,\ \ \ \ a\mapsto a_{(0)}\otimes a_{(1)}.
\end{align}
The subring of invariants of this coaction we denote by $B$, i.e.
\begin{align}
B= A^{{\rm co}\!\!\ C}:=\{ b\in A\mid  b_{(0)}\otimes b_{(1)}=b\otimes e\}.
\end{align}
We assume that $B\subseteq A$ is a $C$-Galois extension, which means that the canonical map 
\begin{align}
can: A\otimes_{B}A\rightarrow A\otimes C,\ \ \ \ 
a\otimes_{B}a'\mapsto aa'_{(0)}\otimes a'_{(1)}
\end{align}
and the canonical entwining
\begin{align}
\psi : 
C\otimes A\rightarrow A\otimes C,\ \ \ \ 
c\otimes a\mapsto {\rm can}({\rm can}^{-1}(1\otimes c)a)
\end{align}
are both invertible, one can define a left $C$-coaction on $A$
\begin{align}
A\rightarrow C\otimes A,\ \ \ \ a\mapsto a_{(-1)}\otimes a_{(0)}:=\psi^{-1}(a_{(0)}\otimes a_{(1)}).
\end{align}
If $C=H$ is a Hopf algebra with an invertible antipode $S$ and the grouplike element $e$ is the unit of $H$, and $A$ is a right $H$-comodule 
algebra, then  the coalgebra-Galois extension is called Hopf-Galois (with the Galois Hopf algebra $H$), and the left $H$-coaction makes the  opposite 
algebra $A^{\rm op}$ a left comodule algebra over $H$.

For a $C$-Galois extension as above one can define the \emph{translation map}  
\begin{align}
\tau: C\rightarrow A\otimes_{B}\! A,\ \ \ \ \tau(c):={\rm can}^{-1}(1\otimes c)
\end{align}
and prove that it is an $C$-bicolinear map, with respect to left and right coactions of $C$ on the left and right tensorand of $A\otimes_{B}\! A$, 
respectively.

%\subsection{Strong connections}
 
A \emph{strong connection}  $\ell$ \cite{bh04} is a unital ($\ell(e)=1\otimes 1$) $C$-bicolinear lifting of the \emph{translation map} $\tau$,
 $$\xymatrixcolsep{4pc}\xymatrix{& A\otimes A\ar[d]\\
C\ar@{-->}[ru]^{\ell}\ar[r]^{\tau}& A\otimes_{B}\! A .\!\!\!\!}
$$ 
One proves  that for a Hopf-Galois extension as above existence of $\ell$ is equivalent to projectivity of $A$ as a left $B$-module right $H$-comodule \cite{bh04}, as well as to faithful flatness of $A$ as a left $B$-module  \cite{scha-schn}. Such Hopf-Galois extensions are called \emph{principal}. Therefore for coalgebra-Galois extensions one focuses on the case of $C$-equivariant projective $C$-Galois extensions, also known as \emph{principal coalgebra-Galois extensions}. It is still sufficient for the existence of a strong connection.  

For our purposes we might use  the definition  of a strong 
connection just as in \cite{bb05}, which does not require unitality, since unitality does play no role in the sequel.  

\subsection{Ehresmann-Schauenburg quantum groupoid}

Provided $A$ is a  faithfully flat as a left $B$-module $C$-Galois extension of $B$, one introduces the Ehresmann-Schauenburg $B$-coring. In terms 
of the \emph{cotensor product} $\Box^{C}$ of right and left $C$-comodules it is a vector space $M:=A\Box^{C}A$. It is canonically a $B$-
bimodule with a canonical $B$-coring structure
\begin{align}
\Delta:  M&\rightarrow M\otimes_{B}M,\ \ \ \ \ \sum_{i}a_{i}\otimes a'_{i}\mapsto \sum_{i}a_{i(0)}\otimes \tau(a_{i(1)}) \otimes  a'_{i},\\
\varepsilon:  M&\rightarrow B,\ \ \ \ \ \ \ \ \ \ \ \ \ \ \ \! \sum_{i}a_{i}\otimes a'_{i}\mapsto \sum_{i}a_{i}a'_{i}.
\end{align}
If $C=H$ is a Hopf algebra, and $A$ is a  faithfully flat as a left $B$-module $H$-Galois extension of $B$ the Ehresmann-Schauenburg coring $M$ is a unital $B^{\rm e}$-subring of   $A^{\rm e}$
 $M=A\Box^{H}A^{\rm op}\subseteq  A^{\rm e}:=A\otimes A^{\rm op}$, and compatibility of the $B$-coring and subring of  $A^{\rm e}$  structures makes it a \emph{quantum groupoid}. 
 
 However, in what follows we would need only the fact that $(M, \varepsilon)$ is a left $B$-module with the augmentation equal to the counit of that coring. 
 
For Hopf-Galois extensions the canonical row extension  corresponding to the counit of the Ehresmann-Schauenburg quantum groupoid 
$\varepsilon :M\rightarrow B$ can be described as a canonically split extension of left $B$-modules 
$$\xymatrixcolsep{2pc}\xymatrix{0 \ar[r] &\Omega^{1}(A)^{{\rm co}\!\!\ H}\ar[r] &(A\otimes A)^{{\rm co}\!\!\ H} \ar[r]_{\varepsilon}&A^{{\rm co}\!\!\ H}\ar@{-->}@/_1.5pc/[l]_{\sigma} \ar[r]& 0}$$
with the canonical splitting  $\sigma(b)=b\otimes 1$ being an algebra map. Here  $\Omega^{1}(A)$ denotes the $A$-bimodule of universal noncommutative differentials of $A$. Therefore we obtain  the following complete description of the block-matrix structure  of our row extension
\begin{align}\label{row}
M\stackrel{\cong}{\rightarrow} \left(\begin{array}{cc}B &\Omega^{1}(A)^{{\rm co}\!\!\ H}\\
              0 & 0\end{array}\right), \ \ \ \ \ 
             \sum_{i}a_{i}\otimes a'_{i}\mapsto \left(\begin{array}{cc} \sum_{i}a_{i}a'_{i} & \sum_{i}a_{i} {\rm d}a'_{i}\\
              0 & 0\end{array}\right),
\end{align}
where on the right-hand side we use the $A$-bimodule structure of $\Omega^{1}(A)$ and the universal derivation 
$${\rm d}: A\mapsto \Omega^{1}(A),\ \ \ {\rm d}a=1\otimes a - a\otimes 1.$$

\subsection{The Chern-Weil homomorphism from a strong connection}

For any coalgebra $C$ we define the subspace 
$$\xymatrix{
C^{\rm tr}:={\rm Eq}(C \ar@<-.3ex>[r] \ar@<.6ex>[r] & C\otimes C)}$$
equalizing the pair of maps, where one arrow is the comultiplication and the second is the multiplication composed with the flip. 

Note that the canonical map $C^{\rm tr}\rightarrow C$ is the dual counterpart of the universal trace map $A\rightarrow A/[A, A]$ for an algebra $A$, and an element  of  $C^{\rm tr}$ defines canonically a trace on the dual convolution algebra $C^{*}$.

The classical case, when $C=\mathscr{O}(G)$ is the coordinate algebra  of a linear algebraic group $G$ with the comultiplication equivalent to the polynomial group composition, sheds some light on the problem of deeper understanding the relation between $C^{\rm tr}$ and the Chern-Weil map as follows. 

First of all, it is easy to see that 
\begin{align}\label{cotr}
C^{\rm tr}= \mathscr{O}({\rm Ad}(G))^{G},
\end{align}
 the latter meaning the algebra of ${\rm Ad}$-invariants (with respect to the action of $G$ on itself by means of conjugations), in other words, the algebra of class functions. 
 
Moreover, the kernel of the augmentation of $C$ is the maximal ideal $\mathfrak{m}$ corresponding to the neutral element of $G$. Moreover, the $\mathfrak{m}$-adic filtration of $\mathscr{O}({\rm Ad}(G))$ is $G$-invariant, hence passing to the ${\rm Ad}$-invariants of the associated graded algebra one gets the algebra
 \begin{align}
{\rm gr}_{\mathfrak{m}}\mathscr{O}\left({\rm Ad}(G)\right)^{G}=\left(\bigoplus_{n\geq 0}\mathfrak{m}^{n}/\mathfrak{m}^{n+1}\right)^{G}\cong {\rm Sym}(\mathfrak{m}/\mathfrak{m}^{2})^{G}=\bigoplus_{n\geq 0}\left({\rm Sym}^{n}\mathfrak{g}^{*}\right)^{G}
\end{align}  
of ${\rm Ad}$-invariant polynomials on the Lie algebra $\mathfrak{g}$. The latter is the domain of the classical Chern-Weil map and the infinitesimal counterpart of the right hand side of (\ref{cotr}). In the opposite direction,  replacing the Lie algebra $\mathfrak{g}$ by the $G$-space ${\rm Ad}(G)$ plays a fundamental role in the construction of $G$-equivariant cyclic homology after Block-Getzler \cite{blogetz}.

Below we will use the associativity of the cotensor product of bicomodules over a coalgebra \cite[11.6]{brze-wisb} and  the tensor-cotensor associativity \cite[10.6]{brze-wisb}, which both hold since our ground ring is a field. 
\begin{lemma}
For any $m\in\mathbb{N}$ the $m$-fold comultiplication map on $C$ defines a linear isomorphism
\begin{align}
C^{\rm tr}\stackrel{\cong}{\rightarrow}C\Box^{C\otimes C^{\rm op}}\left(\underbrace{C\hspace{0.2em}\Box^{C}\cdots \Box^{C}C}_{m+1}\right).
\end{align}
\end{lemma}
\begin{proof}
Since $C$ is the unit object of the monoidal category of bicomodules with respect to the cotensor product, it is enough to prove that isomorphism for  
$m=0$. Then it is easy to check that the comultiplication restricted to the elements from $C^{\rm tr}\subseteq C$ lands in 
$C\Box^{C\otimes C^{\rm op}}C$, and the application of the counit to the first cotensor factor  provides the inverse.
\end{proof}
Note that  applying the counit to the left most factor $C$ in the cotensor product 
\[
C\Box^{C\otimes C^{\rm op}}\left(V_{0}\hspace{0.2em}\Box^{C}\cdots \Box^{C}V_{m}\right)
\]
  we can identify it  with the circular cotensor product. For example, for m = 5 it looks like follows.
$$
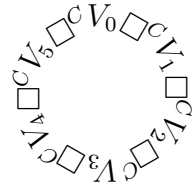
\begin{tikzpicture}[decoration = {markings,
    mark = at position 0.5 with {..........[>=stealth]{>}}
  }
  ]
  \path
         (90 :1cm) node (90) {$V_{0}$}
          (60 :1cm) node (60) {\rotatebox[origin=c]{-30}{$\hspace{0.05em} \Box^{C}\!$}}
         (30 :1cm) node (30) {\rotatebox[origin=c]{-60}{$V_{1}$}}
          (0 :1cm) node (0) {\rotatebox[origin=c]{-90}{$\hspace{0.05em}\Box^{C}\!$}}
        (330:1cm) node (330) {\rotatebox[origin=c]{240}{$V_{2}$}}
         (300 :1cm) node (300) {\rotatebox[origin=c]{210}{$\hspace{0.05em}\Box^{C}\!$}}
        (270:1cm) node (270) {\rotatebox[origin=c]{180}{$V_{3}$}}
        (240 :1cm) node (240) {\rotatebox[origin=c]{150}{$\hspace{0.05em}\Box^{C}\!$}}
        (210:1cm) node (210) {\rotatebox[origin=c]{120}{$V_{4}$}}
        (180 :1cm) node (180) {\rotatebox[origin=c]{90}{$\hspace{0.05em}\Box^{C}\!$}}
        (150:1cm) node (150) {\rotatebox[origin=c]{60}{$V_{5}$}}
       (120 :1cm) node (120) {\rotatebox[origin=c]{30}{$\hspace{0.05em}\Box^{C}\!$}};
\end{tikzpicture}$$
We use circular cotensor products in the following lemma.
\begin{lemma}\label{chw map}
The strong connection $\ell:C\rightarrow A\otimes A$ induces a linear map
\begin{align}
C^{\rm tr}\rightarrow M^{\otimes(n+1)},\ \ \ \ c\mapsto {\rm c}_{m}(\ell)(c):= {\ell}(c_{(1)})\otimes\cdots\otimes {\ell}(c_{(m+1)})
\end{align}
\end{lemma}
\begin{proof}
Since $\ell$ is a morphism of $C$-bicomodules, it can be applied to an element of the circular cotensor power of $C$ to get an element of the circular cotensor power of $A\otimes A$, which by the tensor-cotensor associativity can be written as an element of tensor power of $M=A\hspace{0.2em} \Box^{C}\! A$. For example, ${\rm c}_{5}(\ell)$ reads as  the dashed arrow in the following commutative diagram.
$$\xymatrixcolsep{5pc}
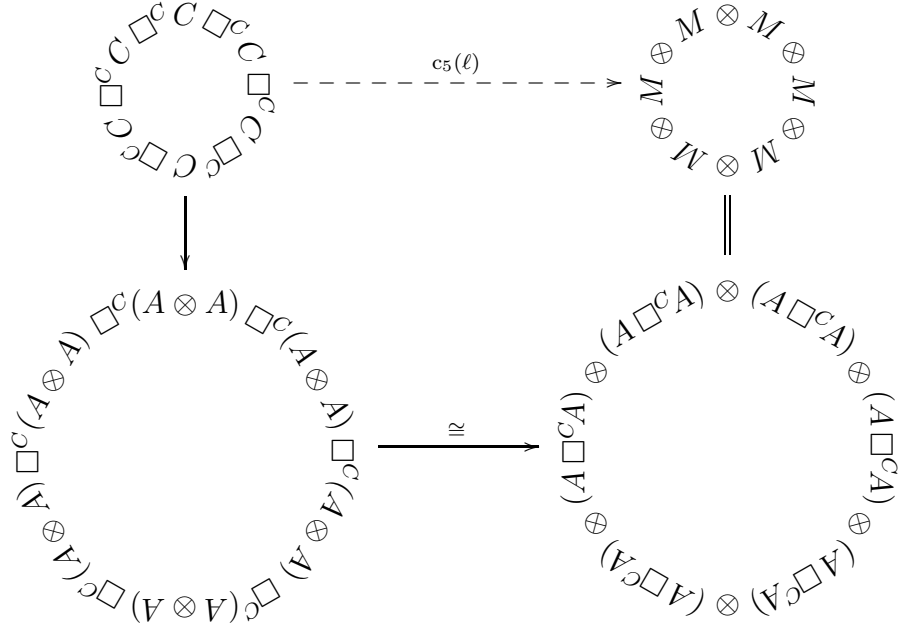
\xymatrix{\begin{tikzpicture}[decoration = {markings,
    mark = at position 0.5 with {..........[>=stealth]{>}}
  }
  ]
  \path
         (90 :1cm) node (90) {$C$}
          (60 :1cm) node (60) {\rotatebox[origin=c]{-30}{$\hspace{0.05em} \Box^{C}\!$}}
         (30 :1cm) node (30) {\rotatebox[origin=c]{-60}{$C$}}
          (0 :1cm) node (0) {\rotatebox[origin=c]{-90}{$\hspace{0.05em}\Box^{C}\!$}}
        (330:1cm) node (330) {\rotatebox[origin=c]{240}{$C$}}
         (300 :1cm) node (300) {\rotatebox[origin=c]{210}{$\hspace{0.05em}\Box^{C}\!$}}
        (270:1cm) node (270) {\rotatebox[origin=c]{180}{$C$}}
        (240 :1cm) node (240) {\rotatebox[origin=c]{150}{$\hspace{0.05em}\Box^{C}\!$}}
        (210:1cm) node (210) {\rotatebox[origin=c]{120}{$C$}}
        (180 :1cm) node (180) {\rotatebox[origin=c]{90}{$\hspace{0.05em}\Box^{C}\!$}}
        (150:1cm) node (150) {\rotatebox[origin=c]{60}{$C$}}
       (120 :1cm) node (120) {\rotatebox[origin=c]{30}{$\hspace{0.05em}\Box^{C}\!$}};
\end{tikzpicture}\ar[dd]\ar@<7ex>@{-->}[r]^{{\rm c}_{5}(\ell)}&    \begin{tikzpicture}
%[decoration = {markings,
%    mark = at position 0.5 with {..........[>=stealth]{>}}
%  }
%  ]
  \path
         (60 :1cm) node (60) {\rotatebox[origin=c]{-30}{$M$}}
          (30 :1cm) node (30) {\rotatebox[origin=c]{-60}{$\otimes$}}
         (0 :1cm) node (0) {\rotatebox[origin=c]{-90}{$M$}}
          (330 :1cm) node (330) {\rotatebox[origin=c]{-120}{$\otimes$}}
        (300:1cm) node (300) {\rotatebox[origin=c]{210}{$M$}}
         (270 :1cm) node (270) {\rotatebox[origin=c]{180}{$\otimes$}}
        (240:1cm) node (240) {\rotatebox[origin=c]{150}{$M$}}
        (210 :1cm) node (210) {\rotatebox[origin=c]{120}{$\otimes$}}
        (180:1cm) node (180) {\rotatebox[origin=c]{90}{$M$}}
        (150 :1cm) node (150) {\rotatebox[origin=c]{60}{$\otimes$}}
        (120:1cm) node (120) {\rotatebox[origin=c]{30}{$M$}}
       (90 :1cm) node (90) {\rotatebox[origin=c]{0}{$\otimes$}};
\end{tikzpicture}\ar@{=}[dd]                             \\ \\
\begin{tikzpicture}[decoration = {markings,
    mark = at position 0.5 with {..........[>=stealth]{>}}
  }
  ]
  \path
         (90 :2cm) node (90) {$(A\otimes A)$}
          (60 :2cm) node (60) {\raisebox{2.2\depth}{{\rotatebox[origin=c]{-30}{$\hspace{0.4em} \Box^{C}\!$}}}}
         (30 :2cm) node (30) {\rotatebox[origin=c]{-60}{$(A\otimes A)$}}
          (0 :2cm) node (0) {\hspace{0.5em}\rotatebox[origin=c]{-90}{$\hspace{0.1em}\Box^{C}\!$}}
        (330:2cm) node (330) {\rotatebox[origin=c]{240}{$(A\otimes A)$}}
         (300 :2cm) node (300) {\raisebox{-1.7\height}{\rotatebox[origin=c]{210}{$\hspace{0.00em}\Box^{C}\!$}}}
        (270:2cm) node (270) {\rotatebox[origin=c]{180}{$(A\otimes A)$}}
        (240 :2cm) node (240) {\raisebox{-1.5\height}{\rotatebox[origin=c]{150}{$\hspace{0.4em}\Box^{C}\!$}}}
        (210:2cm) node (210) {\rotatebox[origin=c]{120}{$(A\otimes A)$}}
        (180 :2cm) node (180) {$\hspace{-0.5em}$\rotatebox[origin=c]{90}{$\hspace{0.05em}\Box^{C}\!$}}
        (150:2cm) node (150) {\rotatebox[origin=c]{60}{$(A\otimes A)$}}
       (120 :2cm) node (120) {\raisebox{0.8\height}{\rotatebox[origin=c]{30}{$\hspace{0.05em}\Box^{C}\!$}}};
\end{tikzpicture}\ar@<13ex>[r]^{\cong}& 
\begin{tikzpicture}[decoration = {markings,
    mark = at position 0.5 with {..........[>=stealth]{>}}
  }
  ]
  \path
         (60 :2cm) node (60) {\rotatebox[origin=c]{-30}{$(A\hspace{0.2em} \Box^{C}\! A)$}}
          (30 :2cm) node (30) {\raisebox{1\depth}{{\rotatebox[origin=c]{-60}{$\otimes$}}}}
         (0 :2cm) node (0) {\rotatebox[origin=c]{-90}{$(A\hspace{0.2em} \Box^{C}\! A)$}}
          (330 :2cm) node (330) {\rotatebox[origin=c]{-120}{$\otimes$}}
        (300:2cm) node (300) {\rotatebox[origin=c]{210}{$(A\hspace{0.2em} \Box^{C}\! A)$}}
         (270 :2cm) node (270) {\raisebox{-1.7\height}{\rotatebox[origin=c]{180}{$\otimes$}}}
        (240:2cm) node (240) {\rotatebox[origin=c]{150}{$(A\hspace{0.2em} \Box^{C}\! A)$}}
        (210 :2cm) node (210) {\rotatebox[origin=c]{120}{$\otimes$}}
        (180:2cm) node (180) {\rotatebox[origin=c]{90}{$(A\hspace{0.2em} \Box^{C}\! A)$}}
        (150 :2cm) node (150) {\rotatebox[origin=c]{60}{$\otimes$}}
        (120:2cm) node (120) {\rotatebox[origin=c]{30}{$(A\hspace{0.2em} \Box^{C}\! A)$}}
       (90 :2cm) node (90) {\raisebox{0.8\height}{\rotatebox[origin=c]{0}{$\otimes$}}};
\end{tikzpicture}}$$
\end{proof}

\begin{lemma}\label{symm cm} 
For any $c\in C^{\rm tr}$ the element 
\begin{align}\label{chw cm}
{\rm c}_{m}(\ell)(c):=
\left(c_{(m+1)}^{\ \ \langle 2\rangle}\otimes c_{(1)}^{\ \ \langle 1\rangle}\right)\otimes \left(c_{(1)}^{\ \ \langle 2\rangle}\otimes c_{(2)}^{\ \ 
\langle 1\rangle}\right)\otimes\cdots\otimes \left(c_{(m)}^{\ \ \langle 2\rangle}\otimes c_{(m+1)}^{\ \ \langle 1\rangle}\right)
\end{align}
belongs to the cyclic symmetric part of $M^{\otimes\!\!\ m+1}$.  
\end{lemma}
\begin{proof} First of all let us note that, by the very definition of $c\in C^{\rm tr}$, applying the comultiplication $\Delta$ to any $c\in C^{\rm tr}$ 
we obtain a symmetric tensor 
\begin{align}\label{delta ctr}
c_{(1)}\otimes c_{(2)}=c_{(2)}\otimes c_{(1)}.
\end{align}
Since by coassociativity the result of application of the iterated comultiplication $\Delta^{m}$ to $c$ is the same as the result of  application of 
$\Delta^{m-1}\otimes C$ to  both sides of (\ref{delta ctr}), we have  
\begin{align}\label{iterated delta ctr}
\begin{split}
&\ \ \ \ c_{(1)}\otimes c_{(2)}\otimes\cdots\otimes c_{(m)}\otimes c_{(m+1)}\\
&= c_{(1)(1)}\otimes c_{(1)(2)}\otimes\cdots\otimes c_{(1)(m)}\otimes c_{(2)}\\
&= c_{(2)(1)}\otimes c_{(2)(2)}\otimes\cdots\otimes c_{(2)(m)}\otimes c_{(1)}\\
&= c_{(2)}\otimes c_{(3)}\otimes\cdots\otimes c_{(m+1)}\otimes c_{(1)},
\end{split}
\end{align}
which proves that the right hand side of (\ref{chw cm})  is a cyclic-symmetric tensor as well.
\end{proof}

\begin{lemma}\label{face}
For any face operator ${\rm d}_{i}$ coming from the multiplication in $M$ the elements ${\rm c}_{m}:={\rm c}_{m}(\ell)(c)$ satisfy the identities
\begin{align}\label{dc}
{\rm d}_{i}{\rm c}_{m}={\rm c}_{m-1}.
\end{align}
\end{lemma}
\begin{proof}
By cyclic symmetry established by Lemma \ref{symm cm} it is enough to check the desired identity only for the 0-th face operator ${\rm d}_{0}$. This goes as follows. 
\begin{align}
{\rm d}_{0}{\rm c}_{m}
&=\left(c_{(m+1)}^{\ \ \langle 2\rangle} c_{(1)}^{\ \ \langle 1\rangle}c_{(1)}^{\ \ \langle 2\rangle}\otimes c_{(2)}^{\ \ \langle 1\rangle}\right)\otimes\cdots\otimes \left(c_{(m)}^{\ \ \langle 2\rangle}\otimes c_{(m+1)}^{\ \ \langle 1\rangle}\right)\\
&=\left(c_{(m+1)}^{\ \ \langle 2\rangle}\varepsilon( c_{(1)})\otimes c_{(2)}^{\ \ \langle 1\rangle}\right)\otimes\cdots\otimes \left(c_{(m)}^{\ \ \langle 2\rangle}\otimes c_{(m+1)}^{\ \ \langle 1\rangle}\right)\\
&=\left(c_{(m+1)}^{\ \ \langle 2\rangle}\otimes \varepsilon( c_{(1)})c_{(2)}^{\ \ \langle 1\rangle}\right)\otimes\cdots\otimes \left(c_{(m)}^{\ \ \langle 2\rangle}\otimes c_{(m+1)}^{\ \ \langle 1\rangle}\right)\\
&=\left(c_{(m)}^{\ \ \langle 2\rangle}\otimes c_{(1)}^{\ \ \langle 1\rangle}\right)\otimes\cdots\otimes \left(c_{(m-1)}^{\ \ \langle 2\rangle}\otimes c_{(m)}^{\ \ \langle 1\rangle}\right)\\
&=\ {\rm c}_{m-1}.
\end{align}
\end{proof}

%Since we can consider the dummy indices $(i_{1},\ldots, i_{n+1})$  in (\ref{iterated}) as cyclically ordered, then applying to it $\ell^{n+1}$ we see that (\ref{cn}) is  cyclic symmetric  as a tensor in $(A\otimes A)^{\otimes\!\!\ n+1}$, where factors which are cyclically permuted belong to $A\otimes A$. Then, applying  Lemma \ref{lem1} to all expressions of the form 
%\begin{align}\label{four}
%\sum_{i_{j}}
%e_{i_{j-1}i_{j}}^{\ \ \ \langle 1\rangle}\otimes 
%\left( e_{i_{j-1}i_{j}}^{\ \ \ \langle 2\rangle}\otimes 
%e_{i_{j}i_{j+1}}^{\ \ \ \langle 1\rangle}\right)\otimes 
%e_{i_{j}i_{j+1}}^{\ \ \ \langle 2\rangle}
%\end{align}
%in  (\ref{cn}) proves that it  is a  cyclic symmetric  tensor in $M^{\otimes\!\!\ n+1}$. To prove (\ref{face}) it is enough to apply  Lemma \ref{lem2} to every expression  of the form  
%\begin{align}\label{six}
%\sum_{i_{j}, i_{j+1}}
%e_{i_{j-1}i_{j}}^{\ \ \ \langle 1\rangle}\otimes 
%\left( e_{i_{j-1}i_{j}}^{\ \ \ \langle 2\rangle}\otimes 
%e_{i_{j}i_{j+1}}^{\ \ \ \langle 1\rangle}\right)\otimes 
%\left(e_{i_{j}i_{j+1}}^{\ \ \ \langle 2\rangle}\otimes e_{i_{j+1}i_{j+2}}^{\ \ \ \langle 1\rangle}\right)\otimes e_{i_{j+1}i_{j+2}}^{\ \ \ \langle 2\rangle}.
%\end{align} 
%\end{proof}
\begin{theorem}[Noncommutative Chern-Weil homomorphism] For any $c\in C^{\rm tr}$ 
\begin{align}
\widetilde{\rm chw}_{n}(\ell, c):=\sum_{m=0}^{2n}(-1)^{\lfloor m/2\rfloor}\frac{m!}{\lfloor m/2\rfloor !}{\rm c}_{m}(\ell)(c)
\end{align}
is a $2n$-cycle in the total complex ${\rm Tot\ CC}_{\bullet}(M)=\bigoplus_{m=0}^{\bullet}M^{\otimes\ \!\! m+1}$ computing cyclic homology  ${\rm HC}_{\bullet}(M)$. Its homology class is stable under Connes' periodicity operator $S$. 
\end{theorem}
\begin{proof}
By Lemma \ref{symm cm} and  Lemma \ref{face} the chains ${\rm c}_{m}(\ell)(c)$ satisfy assumptions of  Proposition \ref{abs ch}, which proves the claim. 
\end{proof} 

Composing with the map induced by the algebra map $\varepsilon : M\rightarrow B$ we obtain the Chern-Weil map ${\rm chw}_{n}(\ell)$ with values in the total complex ${\rm Tot\ CC}_{\bullet}(B)$.  

\subsection{A factorization of the Chern-Galois character}

For any coalgebra coalgebra  $C$ we consider the group completion ${\rm Rep}(C)$ of the monoid of finite dimensional left $C$-comodules, which we call \emph{representations}.
 
If $V$ is a representation then given a basis $(v_{i})_{i\in I}$ of $V$ the left $C$-comodule structure $V\rightarrow C\otimes V$ is equivalent to a finite matrix $(c_{ij})_{i,j\in I}$ with entries in $C$, defined by $v_{i}\mapsto \sum_{j}c_{ij}\otimes v_{j}$ and satisfying
\begin{align}\label{comodmatr}
\Delta(c_{ik})=\sum_{j}c_{ij}\otimes c_{jk},\ \ \varepsilon(c_{ij})=\delta_{ij}.
\end{align}
It is obvious that the element 
\begin{align}\label{rep char}
\chi(V):=\sum_{i}c_{ii}
\end{align} 
is independent of the choice of the basis  and hence depends only on the isomorphism class $[V]$ of $V$. We will call it the \emph{character} of the representation $V$. By the fact that for any short exact sequence of representations 
\begin{align}
0\rightarrow V'\rightarrow V\rightarrow V''\rightarrow 0
\end{align}
one has
\begin{align}
\chi(V')+\chi(V'')=\chi(V)
\end{align}
$\chi$ factorizes through  ${\rm Rep}(C)$. By 
 the obvious symmetry property 
\begin{align}
\chi(V)_{(1)}\otimes \chi(V)_{(2)}=\sum_{i,j}c_{ij}\otimes c_{ji}=\sum_{i,j}c_{ji}\otimes c_{ij}=\chi(V)_{(2)}\otimes \chi(V)_{(1)}
\end{align}
the character  of a representation defines a map 
$$\chi: {\rm Rep}(C)\rightarrow C^{\rm tr},\ \ \ \ [V]\mapsto \chi(V).$$

For the righ-hand side we use the formula from Lemma \ref{symm cm} and the definition (\ref{rep char}) of the character of a representation to compute the composition
\begin{align}\label{cm chi}
{\rm c}_{m}(\ell)(\chi(V)):=\sum_{i_{1},\ldots, i_{m+1}}\left(c_{i_{1}i_{2}}^{\ \!\langle 2\rangle}\otimes c_{i_{2}i_{3}}^{\ \!\langle 1\rangle}\right)\otimes \left(c_{i_{2}i_{3}}^{\ \!\langle 2\rangle}\otimes c_{i_{3}i_{4}}^{\ \!\langle 1\rangle}\right)\otimes\cdots\otimes \left(c_{i_{m+1}i_{1}}^{ \ \! \langle 2\rangle}\otimes c_{i_{1}i_{2}}^{\ \!\langle 1\rangle}\right)
\end{align}
which after applying the map induced by the algebra map $\varepsilon : M\rightarrow B$ is sent to
\begin{align}\label{rhs cm chi}
\sum_{i_{1},\ldots, i_{m+1}}c_{i_{1}i_{2}}^{\ \!\langle 2\rangle} c_{i_{2}i_{3}}^{\ \!\langle 1\rangle}\otimes c_{i_{2}i_{3}}^{\ \!\langle 2\rangle} c_{i_{3}i_{4}}^{\ \!\langle 1\rangle}\otimes\cdots\otimes c_{i_{m+1}i_{1}}^{ \ \! \langle 2\rangle} c_{i_{1}i_{2}}^{\ \!\langle 1\rangle}.
\end{align}
The latter is equal to an expression appearing in the definition of the Chern-Galois character in \cite{bh04}.

\begin{corollary}\label{decomp}
The Chern-Galois character decomposes as the diagonal composition in the following commutative diagram
 $$\xymatrixcolsep{5em}\xymatrix{
{\rm Rep}(C) \ar[d]_{\chi} \ar[r]^-{[A\Box^{C}-]} &{\rm K}_{0}(B)\ar[d]^{{\rm ch}_{n}}\\
C^{\rm tr} \ar[r]^-{{\rm chw}_{n}} &{\rm HC}_{2n}(B).}
$$
\end{corollary}

Besides Corollary~\ref{decomp}, there is another relation between the Chern-Weil map and the  Chern character. It consists in the role played by \
nonunital  $H$-unital block-matrix algebra extensions of the algebra $B$ (even if $B$ is unital and commutative), in defining these maps. 

For the Chern-Weil map, in analogy with the algebra ${\rm M}_{\infty}(B)$ for the  Chern character,  it is the Ehresmann-Schauenburg quantum groupoid (in the Hopf-Galois case) or the Ehresmann-Schauenburg coring 
(in the coalgebra-Galois case)  $M=A\Box^{C}A$ with its multiplication defined by its counit $\varepsilon$. 

%\begin{lemma}\label{lem1}
%Assume that $B\subseteq A$ is a principal Hopf-Galois extension  with a strong connection
%\begin{align}
%\ell: H\rightarrow A\otimes A,\ \ \ h\mapsto h^{\langle 1\rangle}\otimes h^{\langle 2\rangle},
%\end{align}
%where implicit summation on the right hand side is understood, and $M=A\Box^{H}A$. Then, for any $h\in H$ 
%\begin{align}
%{h_{(1)}}^{\!\!\langle 1\rangle}\otimes {h_{(1)}}^{\!\!\langle 2\rangle}\otimes {h_{(2)}}^{\!\!\langle 1\rangle}\otimes {h_{(2)}}^{\!\!\langle 2\rangle}\in A\otimes M\otimes A.
%\end{align}
%\end{lemma}
%\begin{proof}
%Te left hand side can be obtained as the image of $h\in H$ under the canonical composite map
%\begin{align}
%H\stackrel{\Delta}{\longrightarrow}H\ \Box^{H}H\stackrel{\ell\ \Box^{H}\ell}{-\!\!\!-\!\!\!-\!\!\!\longrightarrow}(A\otimes A)\Box^{H}(A\otimes A) =   A\otimes (A\Box^{H}A)\otimes  A.
%\end{align}

\subsection{Independence of the choice of a strong connection}

The fundamental property of the classical Chern-Weil homomorphism is its independence of the choice of a connection. 
As we do not know how to reproduce the classical argument in the noncommutative context, herein we use the independence of
the Chern-Galois character of the choice of a strong connection to argue such independence for the noncommutative Chern-Weil homomorphism.

We will say that $C$ has \emph{enough characters}, if $C^{\rm tr}$ is linearly spanned by characters of representations. Note that the algebra of 
class functions on a semi-simple connected algebraic group has a linear basis consisting of characters of irreducible rational representations 
\cite[3.2]{humph}. The same is true for finite groups. This motivates our terminology. 
\begin{proposition}
If $C$ has enough characters, the Chern-Weil map ${\rm chw}(\ell)$ is independent of the choice of the strong connection $\ell$.
\end{proposition}
\begin{proof}
By the results of \cite{bh04} the Chern-Galois character of a representation $V$ computes the Chern character of a  finitely generated projective 
$B$-module $A\Box^{C}V$ associated through a given representation $V$ with a $C$-Galois extension $B\subseteq A$, and hence the Chern Galois 
character is independent of the choice of the strong connection $\ell$.  In view of Theorem \ref{decomp}, assuming that  $C^{\rm tr}$ is linearly 
spanned by characters of representations, ${\rm chw}(\ell)$ is independent of the choice of the strong connection $\ell$ as well. 
\end{proof}

\section*{Acknowledgements}

The authors are very grateful to Gabriella B\"ohm for her initial work on this paper, to Kaveh Mousavand for his help with calculational experiments,
and to Atabey Kaygun for sharing his insight on Wodzicki's excision.
It is also a pleasure to thank Alexander Gorokhovsky,  Masoud Khalkhali, Ryszard Nest and Hugh Thomas for discussions.
This work was partially supported by NCN grant UMO-2015/19/B/ST1/03098.

\end{document}